\documentclass[a4paper,12pt,draft]{amsart}
\usepackage{amsmath, amsthm, amssymb}
\usepackage{url}
\usepackage{braket}

\theoremstyle{plain}
\newtheorem{thm}{Theorem}[section]
\newtheorem{theorem}[thm]{Theorem}

\newtheorem{lemma}[thm]{Lemma}
\newtheorem{cor}[thm]{Corollary}

\theoremstyle{remark}  
\newtheorem{remark}[thm]{Remark}

\newtheorem{example}[thm]{Example}

\theoremstyle{definition}  

\newcommand{\III}{\mathcal{I}}

\begin{document}

\title[On the action of the toggiling group]{On the action of the toggle group of the Dynkin diagram of type $A$}

\thanks{The first author was partially supported by JSPS KAKENHI Grant Number
JP18K03206.}

\author[Y. Numata]{Yasuhide NUMATA}
\address[Y. Numata]{Department of Mathematics, Shinshu University, Matsumoto,
Japan.}
\email{nu@math.shinshu-u.ac.jp}

\author[Y. Yamanouchi]{Yuiko YAMANOUCHI}
\address[Y. Yamanouchi]{Graduate School of Science and Technology, 
Shinshu University, 
Matsumoto, Japan.}

\begin{abstract}
In this article, we consider involutions, called togglings,
on the set of independent sets of the Dynkin diagram of type $A$,
or a path graph.
We are interested in the action of the subgroup of the symmetric group
of the set of independent sets generated by togglings.
We show that the subgroup coincides with the symmetric group.
\end{abstract}


\keywords{Coxeter groups; Togglings of independent sets; The Fibonacci sequence; Symmetric group; Transitive actions.}
\maketitle
\newcommand{\lparent}{{\boldsymbol(}}
\newcommand{\rparent}{{\boldsymbol)}}
\newcommand{\lvrace}{{\boldsymbol\{}}
\newcommand{\rvrace}{{\boldsymbol\}}}
\newcommand{\ed}[2]{\lvrace #1,#2 \rvrace}

\newcommand{\conjugate}[1]{#1^\top}
\newcommand{\Conjugate}[1]{(#1)^\top}
\newcommand{\fib}{f}
\newcommand{\ttt}{\hat{t}}
\newcommand{\id}{\operatorname{id}}

\newcommand{\Sym}{\mathfrak{S}}
\newcommand{\Alt}{\mathfrak{A}}

\section{Introduction}
In this article,
we are interested in the group generated by operations called toggling.
The operation, toggling, is
originally introduced for the set of order ideals of a poset
in Cameron--Fon-Der-Flaass \cite{MR1356845}.
Let $P$ be a finite poset, $J(P)$ the set of order ideals of $P$.
For an element $p \in P$, we define the map $\tau_p\colon J(P)\to J(P)$ by
\begin{align*}
  \tau_p (I)=
  \begin{cases}
    I\cup \Set{p} & (\text{$p$ is a minimal element of $P\setminus I$} )\\
    I\setminus \Set{p} & (\text{$p$ is a maximum element of $I$} )\\
    I&(\text{otherwise})
  \end{cases}
\end{align*}
for $I\in J(P)$.
The toggle group is a subgroup of the symmetric group of $J(P)$ generated by togglings.
Since we have a bijection from $J(P)$ to the set of antichains of $P$
which maps an order ideal to the maximal elements of the order ideal,
we can regard the toggling as the involution on the set of antichains.
Moreover we can define an analogue of the toggling $\tau_p$ 
as involutions not only on  antichains of a poset $P$
but also on a nice subsets of a finite set with a combinatorial structure.
The toggle groups are studied from the view point of the dynamical combinatrics,
e.g., 
Striker--Williams \cite{MR2950491},
Striker \cite{MR3367300},
Joseph \cite{MR3919614}, and
Joseph--Roby \cite{MR4145986}.

In this article,
we consider togglings on 
the set of 
the independent sets,
i.e.,
subsets of vertices such that no pair are adjacent in the graph, 
of the Dynkin diagram of type $A$,
i.e., a path graph. 
In \cite{MR3761932},
Joseph and Roby study
the togglings 
from the viewpoint of
 the dynamical combinatrics.
They study 
the orbit structure and 
the phenomenon called homomecy of the togglings.

In this article,
we are interested in the transitivity of the action
of the toggle group on the set of independent sets.
For a nonempty independent set $I$,
the number of vertices in the independent set $I$
is greater than
the number of vertices in
the resulting independent set $\tau_v(I)$ by the toggling with respect
to a vertex in the independent set $I$.
Hence we can obtain the empty independent set
from $I$ by applying the product of the togglings with respect to all vertices in the independent set.
It follows from the observation  that
the action 
is $1$-transitive.
Multiple transitivity, however, does not seem trivial.
In this paper, we consider the toggle group of the independent sets of the Dynkin diagram of $A_{n}$,
and show that the action of the toggle group is $\fib_{n+2}$-transitive,
where $\fib_{n+2}$ is
the $n+2$-th Fibonacci number, i.e.,
the number of independent sets of  the Dynkin diagram of $A_{n}$.
In other words,
we show that the toggle group coincides with the symmetric group
on the set of independent sets of  the Dynkin diagram of $A_{n}$.

This article is organized as follows:
In Section \ref{sec:sym},
we consider the family of the symmetric groups indexed by the Fibonacci sequence.
We give systems of generators for them as Theorem \ref{thm:gen}.
In Section \ref{sec:toggle},
we recall the definition of the toggling on independent sets of a graph,
and we consider the  toggle group
for independent sets of the Dynkin diagram of type $A$.
As Theorem \ref{thm:togglegroupissymmetricgroup},
we state that the toggle groups for independent sets of the Dynkin diagrams of type $A$
are isomorphic to the symmetric groups indexed by  the Fibonacci sequence.
We give  proofs of main theorems in Section \ref{sec:proof}.

\section{Systems of generators for symmetric groups indexed by the Fibonacci sequence}
\label{sec:sym}
Let $\Set{\fib_n}_{n=0,1,\ldots}$ be the Fibonacci sequence, i.e., 
the sequence of numbers defined by $\fib_0=0$, $\fib_1=1$, and $\fib_n=\fib_{n-1}+\fib_{n-2}$.
We define $F_{n+2}$ to be the set $\Set{1,2,\ldots, \fib_{n+2}}$.
We decompose $F_{n+2}$ into 
$F_{n+1}$
and $\hat F_{n}$, 
where
\begin{align*}
\hat F_{n}
&=F_{n+2}\setminus F_{n+1}\\
&=\Set{1+\fib_{n+1},2+\fib_{n+1},\ldots,\fib_{n}+\fib_{n+1}=\fib_{n+2}}\\
&=\Set{i+\fib_{n+1}|i\in F_{n}}.
\end{align*}

Here we consider the family $\Set{\Sym_{\fib_n}}_{n=3,4,\ldots}$ of the
symmetric groups indexed by the Fibonacci sequence.
For $n\leq m$, we regard 
$\Sym_n$ as a subset of $\Sym_m$ in the usual manner.
In this article, for a set $X$,
$\Sym_X$ stands for the symmetric group on $X$.
Under the notation, $\Sym_{\fib_n}=\Sym_{F_n}$.

We define $\ttt_n$ as follows:
For $n=1,2,\ldots$,
we define
$\ttt_n \in \Sym_{\fib_{n+2}}$ to be the product
\begin{align*}
 (1,\fib_{n+1}+1)(2,\fib_{n+1}+2)\cdots (\fib_{n},\fib_{n+1}+\fib_{n}) 
\end{align*}
of $\fib_{n}$ transpositions.
For $n\geq 1$,
the inner automorphism by $\ttt_n$ induces
the isomorphism
\begin{align*}
 \Sym_{F_{n}} &\to \Sym_{\check F_{n}}\\
t &\mapsto \ttt_n t \ttt_n^{-1}.
\end{align*}
Hence the inner automorphism by $\ttt_n$ also induces
the isomorphism
\begin{align*}
 \Sym_{\fib_{n}} &\to 
\tilde \Sym_{\fib_{n}}\\
t &\mapsto t\cdot \ttt_n t \ttt_n^{-1},
\end{align*}
where
\begin{align*}
\tilde \Sym_{\fib_{n}} &=
\Set{g \in \Sym_{\fib_{n+2}}|
\begin{array}{c}
\forall i \leq \fib_{n},\ g(i+\fib_{n+1})=g(i)+\fib_{n+1}  \\
\fib_{n} < \forall i \leq \fib_{n+1},\ g(i)=i 
\end{array}
}  \\
&\subset \Sym_{F_{n}}\times \Sym_{\hat F_{n}}
\subset \Sym_{F_{n+1}}\times \Sym_{\hat F_{n}}
\subset \Sym_{F_{n+2}}.
\end{align*}
For $1\leq k \leq n$,
we define $t_{k,n}\in \Sym_{F_{n+2}}$ by
\begin{align*}
 t_{k,n}
=\begin{cases}
t_{k,n-1} \cdot \ttt_n t_{k,n-2} \ttt_n^{-1} & (k\leq n-2)\\  
t_{n-1,n-1} & (k=n-1)\\  
\ttt_n & (i=n), 
 \end{cases}
\end{align*}
recursively.
By definition, for $k<n$,
$t_{k,n}$ is an element of $\Sym_{F_{n+1}}\times \Sym_{\hat F_{n}}$,
and $t_{n-1,n}$ is an element of $\Sym_{F_{n+1}}$.
We define subsets $G_{n}$ and $G'_{n}$ of $\Sym_{F_{n+2}}$ by
\begin{align*}
 G'_n&=\Set{t_{k,n}|k\leq n-2}, \\
 G_n&=\Set{t_{k,n}|k\leq n}.
\end{align*}
\begin{example}
\label{basecase:1}
 Consider the case where $n=1$.
Since  $\ttt_1$ is the transposition $(1,2)\in \Sym_{\fib_3}=\Sym_2$,
we have $t_{1,1}=\ttt_1=(1,2)$ and $G_1=\Set{(1,2)}$.
Hence $G_1$ generates $\Sym_2=\Sym_{\fib_3}$.
\end{example}
\begin{example}
\label{basecase:2}
In the case where $n=2$,
 $\ttt_2$ is the transposition $(1,3)\in \Sym_{\fib_4}=\Sym_3$,
we have $t_{1,2}=t_{1,1}=(1,2)$ and $t_{2,2}=\ttt_2=(1,3)$.
Hence $G_2=\Set{(1,2),(1,3)}$.
The set $G_2$ generates $\Sym_3=\Sym_{\fib_4}$.
\end{example}
\begin{example}
\label{basecase:3}
In the case where $n=3$,
 $\ttt_4$ is the product $(1,4)(2,5)\in \Sym_{\fib_5}=\Sym_5$ of the transpositions,
we have
\begin{align*}
 t_{1,3}&=(1,2)\cdot (4,5), \\
 t_{2,3}&=t_{2,2}=(1,3), \\
 t_{3,3}&=\ttt_3=(1,4)(2,5). 
\end{align*}
Hence 
\begin{align*}
 G_3&=\Set{(1,2)(4,5),(1,3),(1,4)(2,5)},\\
 G'_3&=\Set{(1,2)(4,5)}. 
\end{align*}
Since
\begin{align*}
 \tilde \Sym_{\fib_{3}} 
&=
\Set{g \in \Sym_{\fib_{5}}|
\begin{array}{c}
\forall i \leq \fib_{3},\ g(i+\fib_{4})=g(i)+\fib_{4}  \\
\fib_{3} < \forall i \leq \fib_{4},\ g(i)=i 
\end{array}}\\
&=
\Set{g \in \Sym_{\fib_{5}}|
g(1+3)=g(1)+3, g(2+3)=g(2)+3,g(3)=3 
},
\end{align*}
the group $ \tilde \Sym_{\fib_{3}} $ is generated by $\Set{(1,2)(4,5)}$.
It follows from direct calculation that
\begin{align*}
 (1,3)&=t_{2,3},\\
 (2,3)&=t_{1,3} (1,3)t_{1,3}^{-1},\\
 (4,3)&=t_{3,3} (1,3)t_{3,3}^{-1},\\
 (5,3)&=t_{1,3} (4,3)t_{1,3}^{-1}.
\end{align*}
Hence $G_3$ generates $\Sym_5=\Sym_{\fib_5}$. 
\end{example}
\begin{example}
Consider the case where $n=4$.
In this case,
 $\ttt_4=(1,6)(2,7)(3,8)\in \Sym_{\fib_6}=\Sym_8$ of the transpositions.
Hence 
\begin{align*}
t_{1,4}&=(1,2)(4,5)\cdot (6,7), \\
t_{2,4}&=(1,3)\cdot (6,8), \\
t_{3,4}&=t_{3,3}=(1,4)(2,5), \\
t_{4,4}&=\ttt_4=(1,6)(2,7)(3,8). 
\end{align*}
Hence
\begin{align*}
G_4&=\Set{(1,2)(4,5)(6,7), (1,3)(6,8), (1,4)(2,5), (1,6)(2,7)(3,8)},\\
G'_4&=\Set{(1,2)(4,5)(6,7), (1,3)(6,8)}. 
\end{align*}
\end{example}

We show the following theorems by induction on $n$ in Subsection \ref{sec:proof:sym}.
\begin{theorem}
\label{thm:gen'}
For $n=3,4,\ldots$,
the set $G'_n$ generates the group $\tilde \Sym_{\fib_{n}}$.
Hence the group 
$\Braket{G'_n}$
 generated by $G'_n$
 is isomorphic to $\Sym_{\fib_{n}}$.
\end{theorem}
\begin{theorem}
  \label{thm:gen}
For $n=1,2,\ldots$,
the set $G_n$ generates the $\fib_{n+2}$-th symmetric group $\Sym_{\fib_{n+2}}$.
\end{theorem}

\section{The toggle group for independent sets of the Dynkin diagram of type $A$}
\label{sec:toggle}
In this article, we consider simple graphs.
We regard a graph as the pair $(V,E)$ of the set $V$ of vertices and the
set $E$ of edges.
We also regard an edge as a subset of $V$ of size two.
An independent set of a graph is 
a subset  $I$ of the vertex set $V$ of the graph satisfying
\begin{align*}
 u,v\in I \implies \ed{u}{v}\not\in E.
\end{align*}
Roughly speaking,
an independent set of a graph is a subset of vertices 
such that any pair are not adjacent in $(V,E)$.

Let $\III$ be the set of independent sets of a graph $(V,E)$.
For a vertex $v\in V$, we define 
the map $\tau_v$ 
 by
\begin{align*}
\tau_v \colon \III &\to \III\\
I&\mapsto
\begin{cases}
I\setminus\Set{v} & (\text{$v\in I$}) \\
I\cup\Set{v} & (\text{$v\not\in I$ and $I\cup\Set{v} \in\III$}) \\
I &(\text{otherwise}).
\end{cases}
\end{align*}
By definition,
the map $\tau_v$ is an involution on $\III$.
Hence $\tau_v$ is an element of the symmetric group 
$\Sym_{\III}$ on $\III$.
We call the map $\tau_v$ the \emph{toggling}  on $\III$ with respect to $v$.
We also call the subgroup generated by all togglings
the \emph{toggle group} of the graph $(V,E)$.
\begin{remark}
The toggle group acts on the set $\III$ of independent sets naturally.
Moreover, for an independent set $I=\Set{v_1,v_2,\ldots,v_l}$ of size $l$,
we have
\begin{align*}
 \tau_{v_1}\circ \tau_{v_2}\circ \cdots \circ \tau_{v_l}(I) = \emptyset.
\end{align*}
 We can obtain the empty set from any independent set by applying some togglings.
 Hence the group generated by all togglings of a graph acts transitively
 on the set of independent sets of the graph.
 On the other words, the action of the toggle group is $1$-transitive.
\end{remark}

Let $A_{n}=(V_n,E_n)$ be the graph such that
\begin{align*}
 V_n&=\Set{1,\ldots,n},\\
 E_n&=\Set{\ed{1}{2},\ed{2}{3}\ldots,\ed{n-1}{n}},
\end{align*}
i.e., the Dynkin diagram of type $A_{n}$.
Let $\III_n$ be the set of independent sets of $A_n$.
In this article, we are interested in the togglings of the graph $A_{n}$.

For $1\leq k \leq n$,
we define $\tau_{k,n}$ to be 
the toggling $\tau_{k}$ on $\III_n$ with respect to $k\in V_n$,
i.e., 
\begin{align*}
\tau_{k,n} \colon \III_n &\to \III_n\\
I&\mapsto
\begin{cases}
I\setminus \Set{k} & (k \in I)\\
I\cup \Set{k} & (k-1,k,k+1 \not\in I)\\
I&(\text{otherwise}).
\end{cases}
\end{align*}
We also define define $\Gamma_n$ to be the group generated by togglings
\begin{align*}
\Set{\tau_{1,n},\tau_{2,n},\ldots,\tau_{n,n}}.
\end{align*}
We call the group the \emph{toggle group of $A_{n}$}.
By definition the toggle group $\Gamma_n$ of $A_{n}$ acts on the set
$\III_n$ of independent sets on $A_{n+1}$, naturally.
In the other words,
the toggle group $\Gamma_n$ is 
a subgroup of the symmetric group $\Sym_{\III_n}$ of $\III_n$.
\begin{remark}
The togglings $\tau_{k,n}$ for $1\leq k\leq n $ 
satisfy the following relations:
\begin{align*}
 \tau_{k,n}^2&=\id. \\
 \tau_{k,n}\tau_{k',n}&=\tau_{k',n}\tau_{k,n} \quad (|k-k'|>1).\\
 (\tau_{k,n}\tau_{k+1,n})^6 &=\id.
\end{align*}
Hence the toggle group $\Gamma_n$ is a finite quotient group of
the Coxeter group with respect to 
\begin{align*}
\underbrace{\bullet \frac{6}{}\bullet \frac{6}{}\bullet \frac{6}{} \cdots  \frac{6}{}\bullet}_{n},
\end{align*}
which is infinite if $n>2$.
\end{remark}
We will show the following theorem in Subsection \ref{sec:proof:tog}.
\begin{theorem}
 \label{thm:togglegroupissymmetricgroup}
  For $n=1,2,\ldots$,
 the toggle group $\Gamma_n$ of $A_{n}$ is  
the symmetric group $\Sym_{\III_n}$ of the set $\III_n$ of independent sets of $A_{n}$.
Hence 
the toggle group $\Gamma_n$ of $A_{n}$ is
isomorphic to the $\fib_{n+2}$-th symmetric group $\Sym_{\fib_{n+2}}$.
\end{theorem}
\begin{cor}
For $n=1,2,\ldots$,
 the action of the toggle group $\Gamma_n$ of $A_{n+1}$ on $\III_n$ is 
 $\fib_{n+2}$-transitive.
\end{cor}

\section{Proof of Main Theorems}
\label{sec:proof}
\subsection{Theorems \ref{thm:gen'} and \ref{thm:gen}}
\label{sec:proof:sym}
Here we show Theorems \ref{thm:gen'} and \ref{thm:gen}
by induction on $n$.
To show theorems,
we show some lemmas.
\begin{lemma}
\label{lem:g':step}
If $n\geq 3$ and $G_{n-2}$ generates $\Sym_{\fib_{n}}$,
then
$G'_n$ generates $\tilde \Sym_{\fib_{n}}$.
\end{lemma}
\begin{proof}
Consider the isomorphism
\begin{align*}
\varphi\colon \Sym_{\fib_{n}} &\to 
\tilde \Sym_{\fib_{n}}\\
t &\mapsto t\cdot \ttt_n t \ttt_n^{-1}.
\end{align*}
By definition, for $k\leq n-2$,
$t_{k,n}$ is the image $\varphi(t_{k,n-2})$ of $t_{k,n-2}$.
Hence we have
\begin{align*}
 \Braket{G'_n}
&=\Braket{\varphi(t_{1,n-2}),\varphi(t_{2,n-2}),\ldots,\varphi(t_{n-2,n-2})}\\
&=\varphi(\Braket{t_{1,n-2},t_{2,n-2},\ldots,t_{n-2,n-2}})\\
&=\varphi(\Braket{G_{n-2}}).
\end{align*}
Since $G_{n-2}$ generates $\Sym_{\fib_n}$,
we have
$\Braket{G'_n}=\varphi(\Sym_{\fib_n})=\tilde \Sym_{\fib_n}$.
\end{proof}

\begin{lemma}
\label{lem:g:alt}
If $n\geq 4$ and
$G'_n$ generates $\tilde \Sym_{\fib_{n}}$,
then 
the group $\Braket{G_n}$ generated by $G_n$
contains 
cyclic permutations
$(i,i+1,i+2)$ for $1\leq i \leq \fib_{n+2}-2$.
\end{lemma}
\begin{proof}
We decompose $F_{n+2}$ into
three subsets
\begin{align*}
 &F_{n}=\Set{1,2,\ldots,\fib_n},\\
 &\hat F_{n-1}=\Set{i+\fib_{n}|i\in F_{n-1}},\\
 &\hat F_{n}=\Set{i+\fib_{n+1}|i\in F_{n}}.
\end{align*}

First we show that $(1,2,3)\in \Braket{G_n}$.
Since $n\geq 4$, $\fib_n>\fib_3=2$.
Since $\tilde \Sym_{\fib_{n}}=\Braket{G'_n} \subset \Braket{G_n}$,
the group $\Braket{G_n}$ contains the following elements:
\begin{align*}
(1,2)\ttt_n(1,2)\ttt_n^{-1}&=(1,2)(\fib_{n+1}+1,\fib_{n+1}+2),\\
(2,3)\ttt_n(2,3)\ttt_n^{-1}&=(2,3)(\fib_{n+1}+2,\fib_{n+1}+3). 
\end{align*}
Since $t_{n-1,n}=\ttt_{n-1}\in G_n$,
we have
\begin{align*}
\ttt_{n-1}(1,2)(\fib_{n+1}+1,\fib_{n+1}+2)\ttt_{n-1}^{-1}\in \Braket{G_n}. 
\end{align*}
Since $\ttt_{n-1}\in \Sym_{F_{n+1}}$,
\begin{align*}
&\ttt_{n-1}(1,2)(\fib_{n+1}+1,\fib_{n+1}+2)\ttt_{n-1}^{-1}\\
&=
\ttt_{n-1}(1,2)\ttt_{n-1}^{-1}\cdot (\fib_{n+1}+1,\fib_{n+1}+2)\\
&=
(f_{n}+1,f_{n}+2)(\fib_{n+1}+1,\fib_{n+1}+2).
\end{align*}
Hence $\Braket{G_n}$ contains
\begin{align*}
&(f_{n}+1,f_{n}+2)(\fib_{n+1}+1,\fib_{n+1}+2)
\cdot
(2,3)(\fib_{n+1}+2,\fib_{n+1}+3)\\
&=
(2,3)\cdot 
(f_{n}+1,f_{n}+2)
\cdot
(\fib_{n+1}+1,\fib_{n+1}+2)(\fib_{n+1}+2,\fib_{n+1}+3)\\
&=
(2,3)\cdot 
(f_{n}+1,f_{n}+2)
\cdot
(\fib_{n+1}+1,\fib_{n+1}+2,\fib_{n+1}+3).
\end{align*}
Hence $\Braket{G_n}$ contains the square
\begin{align*}
&((2,3)\cdot 
(f_{n}+1,f_{n}+2)
\cdot
(\fib_{n+1}+1,\fib_{n+1}+2,\fib_{n+1}+3))^2\\
&=
(2,3)^2\cdot 
(f_{n}+1,f_{n}+2)^2
\cdot
(\fib_{n+1}+1,\fib_{n+1}+2,\fib_{n+1}+3)^2\\
&=(\fib_{n+1}+1,\fib_{n+1}+3,\fib_{n+1}+2)
\end{align*}
of the element.
Since $\ttt_n (\fib_{n+1}+1,\fib_{n+1}+3,\fib_{n+1}+2) \ttt_n^{-1}=(1,3,2)$,
$\Braket{G_n}$ contains
cyclic permutations $(1,3,2)$ and $(1,2,3)$.

Next we show that 
$\Braket{G_n}$
contains
the cyclic permutation $(i,j,k)$
for $i,j,k\in F_n$.
Fix an element $\sigma \in \Sym_{F_b}$ 
such that $\sigma(1)=i$, $\sigma(2)=j$, $\sigma(3)=k$.
Since 
$\tilde \Sym_{\fib_{n}}=\Braket{G'_n}\subset \Braket{G_n}$,
$\Braket{G_n}$ 
contains
$\tilde\sigma=\sigma\ttt_n\sigma \ttt_n^{-1} \in \tilde \Sym_{\fib_{n}}$.
Hence 
$\Braket{G_n}$ 
contains
\begin{align*}
\tilde\sigma(1,2,3)\tilde\sigma^{-1}
&=(\tilde\sigma(1),\tilde\sigma(2),\tilde\sigma(3))
=(\sigma(1),\sigma(2),\sigma(3))
=(i,j,k).
\end{align*}

Next we show that 
$\Braket{G_n}$
contains
the cyclic permutation $(i,j,k)$
for $i,j,k\in \hat F_{n-1}=\Set{i'+\fib_n|i'\in F_{n-1}}$.
Since 
$F_{n-1}$ contains
$i-\fib_n$,  $j-\fib_n$ and  $k-\fib_n$ 
for $i,j,k\in \hat F_{n-1}$,
$\Braket{G_n}$ contains $(i-\fib_n,j-\fib_n,k-\fib_n)$.
Hence $\Braket{G_m}$ contains
\begin{align*}
&\ttt_{n-1}(i-\fib_n,j-\fib_n,k-\fib_n) \ttt_{n-1}^{-1}\\
&=
(\ttt_{n-1}(i-\fib_n),\ttt_{n-1}(j-\fib_n))\ttt_{n-1}(k-\fib_n))\\
&=
(i,j,k).
\end{align*}

Next we show that 
$\Braket{G_n}$
contains
the cyclic permutation $(i,j,k)$
for $i,j,k\in \hat F_{n}=\Set{i+\fib_{n+1}|i\in F_{n}}$.
Since 
$F_{n}$ contains
$i-\fib_{n+1}$,  $j-\fib_{n+1}$ and  $k-\fib_{n+1}$ 
for $i,j,k\in \hat F_{n}$,
$\Braket{G_n}$ contains $(i-\fib_{n+1},j-\fib_{n+1},k-\fib_{n+1})$.
Hence $\Braket{G_n}$ contains
\begin{align*}
&\ttt_{n}(i-\fib_{n+1},j-\fib_{n+1},k-\fib_{n+1}) \ttt_{n}^{-1}\\
&=
(\ttt_{n}(i-\fib_{n+1}),\ttt_{n-1}(j-\fib_{n+1}))\ttt_{n}(k-\fib_{n+1}))\\
&=
(i,j,k).
\end{align*}

Next 
$\Braket{G_n}$
contains
the cyclic permutation $(i,j,k)$
for $i\in F_n$ and $j,k\in \hat F_{n-1}=\Set{i'+\fib_{n}|i'\in F_{n-1}}$.
For $j,k\in \hat F_{n-1}$,
we have
$j-\fib_n$ and $k-\fib_n\in F_{n-1}\subset F_{n}$.
Since $f_{n} \in F_{n}\setminus F_{n-1}$,
$\Braket{G_n}$ contains
the cyclic permutation $(f_n,j-\fib_n,k-\fib_n)$.
Hence 
$\Braket{G_n}$ contains
\begin{align*}
&\ttt_{n-1} (f_n,j-\fib_n,k-\fib_n)  \ttt_{n-1}^{-1}\\
&=
 (\ttt_{n-1}(f_n),\ttt_{n-1}(j-\fib_n),\ttt_{n-1}(k-\fib_n))\\
&=(f_n,j,k).
\end{align*}
Since $\tilde \Sym_{\fib_{n}}=\Braket{G'_n} \subset \Braket{G_n}$,
$\Braket{G_n}$ contains
\begin{align*}
\tilde\sigma=(i,f_n)\ttt_n(i,f_n)\ttt_n^{-1}=(i,f_n)(\fib_{n+1}+i,\fib_{n+1}+f_n) 
\end{align*}
for $i<f_n$.
Hence $\Braket{G_n}$ contains
\begin{align*}
\tilde\sigma (f_n,j,k) \tilde\sigma^{-1}
=(\tilde\sigma(f_n),\tilde\sigma(j),\tilde\sigma(k))
=(i,j,k).
\end{align*}

Next 
$\Braket{G_n}$
contains
the cyclic permutation $(i,j,k)$
for
$i,j\in \hat F_{n-1}=\Set{i'+\fib_{n}|i'\in F_{n-1}}$ and
$k\in \hat F_{n}=\Set{i+\fib_{n+1}|i\in F_{n}}$.
Since $k-\fib_{n+1}\in F_n$
for $k\in\hat F_{n}$, 
$\Braket{G_n}$
contains
the cyclic permutation  $(k-\fib_{n+1},i,j)$.
Hence
$\Braket{G_n}$
contains
\begin{align*}
\ttt_n (k-\fib_{n+1},i,j)\ttt_n^{-1}
&=
(\ttt_n(k-\fib_{n+1}),\ttt_n(i),\ttt_n(j))\\
&=
(k,i,j)=(i,j,k).
\end{align*}

Next 
$\Braket{G_n}$
contains
the cyclic permutation $(i,j,k)$
for $i,j\in F_n$ and 
$k\in \hat F_{n-1}=\Set{i'+\fib_{n}|i'\in F_{n-1}}$.
Since 
$F_{n-1}$ contains $k-\fib_{n}$ for $k\in\hat F_{n-1}$
and
$\hat F_{n-1}$ contains $\fib_n+1,\fib_n+2$,
$\Braket{G_n}$ contains
$(k-\fib_{n}, \fib_n+1,\fib_n+2)$.
Hence
$\Braket{G_n}$ contains
\begin{align*}
& \ttt_{n-1} (k-\fib_{n}, \fib_n+1,\fib_n+2)  \ttt_{n-1}^{-1}\\
&=(\ttt_{n-1}(k-\fib_{n}),\ttt_{n-1}( \fib_n+1),\ttt_{n-1}(\fib_n+2))\\
&=(k,1,2).
\end{align*}
Fix an element $\sigma \in \Sym_{\fib_n}$
such that $\sigma(1)=i$ and  $\sigma(2)=j$.
Since $\tilde \Sym_{\fib_{n}}=\Braket{G'_n} \subset \Braket{G_n}$,
$\Braket{G_n}$ contains
$\tilde\sigma=\sigma\ttt_n \sigma \ttt_n^{-1}$.
Hence
$\Braket{G_n}$ contains
\begin{align*}
 \tilde\sigma (k,1,2)  \tilde\sigma^{-1}
=( \tilde\sigma(k), \tilde\sigma(1), \tilde\sigma(2))
=(k,i,j)=(i,j,k).
\end{align*}

Next 
$\Braket{G_n}$
contains
the cyclic permutation $(i,j,k)$
for
$i\in \hat F_{n-1}=\Set{i'+\fib_{n}|i'\in F_{n-1}}$ and
$j,k\in \hat F_{n}=\Set{i+\fib_{n+1}|i\in F_{n}}$.
Since
$F_n$ contains
 $j-\fib_{n+1}$
 $k-\fib_{n+1}$
for $j,k\in\hat F_{n}$, 
$\Braket{G_n}$
contains
the cyclic permutation  $(j-\fib_{n+1},k-\fib_{n+1},i)$.
Hence
$\Braket{G_n}$
contains
\begin{align*}
&\ttt_n (j-\fib_{n+1},i-\fib_{n+1},i)\ttt_n^{-1}\\
&=
(\ttt_n(j-\fib_{n+1}),\ttt_n(k-\fib_{n+1}),\ttt_n(i))\\
&=
(j,k,i)=(i,j,k).
\end{align*}

Since $n\geq 4$, $\#\hat F_{n-1}=\fib_{n-1}\geq\fib_3=2$.
Hence we have $(i,i+1,i+2)\in \Braket{G}$ for $1\leq i \leq \fib_{n+2}-2$.
\end{proof}

\begin{lemma}
\label{lem:g:step}
If $n\geq 4$ and
$G'_n$ generates $\tilde \Sym_{\fib_{n}}$,
then 
$G_n$ generates $\Sym_{F_{n+2}}$.
\end{lemma}
\begin{proof}
Since
$\Braket{G_n}$
 contains 
cyclic permutations
$(i,i+1,i+2)$ for $1\leq i \leq \fib_{n+2}-2$ by Lemma \ref{lem:g:alt},
$\Braket{G_n}$
contains 
the $\fib_{n+2}$-th alternating group $\Alt_{\fib_{n+2}}$.
Since either $f_{n}$ or $f_{n-1}$
is an odd number,
the alternating group  $\Alt_{\fib_{n+2}}$
does not contain either $t_{n,n}=\ttt_n$ or $t_{n,n-1}=\ttt_{n-1}$.
Since the alternating group $\Alt_{\fib_{n+2}}$ is
a subgroup of $\Sym_{\fib_{n+2}}$ of index $2$,
the subgroup $\Braket{G_n}$ coincide with $\Sym_{\fib_{n+2}}$.
\end{proof}

By Examples \ref{basecase:1}, \ref{basecase:2} and \ref{basecase:3},
we can show the base case, i.e., the case where $n=1,2,3$.
Moreover we have Lemmas \ref{lem:g':step} and \ref{lem:g:step}
for induction step.
Hence we have Theorems \ref{thm:gen'} and \ref{thm:gen}.

\subsection{Proof of Theorem \ref{thm:togglegroupissymmetricgroup}}
\label{sec:proof:tog}
To show Theorem \ref{thm:togglegroupissymmetricgroup},
here we give an bijection between $\III_n$ and $F_{n+2}$,
and translate togglings to permutations in $\Sym_{f_{n+2}}$.

We can  decompose the set $\III_n$ of 
independent set of $A_n$
into the following two subsets:
\begin{align*}
\Set{I\in \III_n| n\not\in I },\\
\Set{I\in \III_n| n\in I}. 
\end{align*}
It is easy to see that 
$\Set{I\in \III_n| n\not\in I } =  \III_{n-1}$.
We define $\hat \III_{n-2}$ to be the set 
$\Set{I\in \III_n| n\in I}$.
For $I\in \III_n$, by definition,
we have $n-1 \not\in I$ if $n\in I$.
Hence we have the bijection $\varphi_{n-2}$
from $\hat \III_{n-2}$ to $\III_{n-2}$
by removing the vertex $n$.
i.e., the map defined by
\begin{align*}
 \varphi_{n-2} \colon \hat \III_{n-2} &\to \III_{n-2}\\
 I&\mapsto I\setminus\Set{n}.
\end{align*}
Hence we have the recurrence relation  $\# \III_n=\# \III_{n-1}+\# \III_{n-2}$.
Since $\# \III_1=2$ and $\# \III_2=3$,
the number $\# \III_n$ of independent sets of $A_{n}$
is equal to the $(n+2)$-th Fibonacci number $\fib_{n+2}$.

Since the number of $\III_n$ is $\fib_{n+2}=\# F_{n+2}$,
the elements of $\III_n$ can be indexed by $F_{n+2}$.
We give the index $\iota_n(I)$ of an independent set $I$ in $\III_n$ in the
following manner:
In the case  where $n=1$, we define
\begin{align*}
\iota_1(I) =
\begin{cases}
 1 &(I=\emptyset)\\
 2 &(I=\Set{1}).
\end{cases}
\end{align*}
In the case  where $n=2$, we define
\begin{align*}
\iota_1(I) =
\begin{cases}
 1 &(I=\emptyset)\\
 2 &(I=\Set{1})\\
 3 &(I=\Set{2}).
\end{cases}
\end{align*}
In the case where $n>2$, we define
\begin{align*}
 \iota_n(I)=
\begin{cases}
\iota_{n-1}(I) &(I \in \III_{n-1}) \\
\iota_{n-2}( \varphi_{n-2}(I) ) + \fib_{n+1}&(I \in \hat \III_{n-2}),
\end{cases}
\end{align*}
recursively.
By definition,
the index of an independent set in $\III_{n-1}$ is in $F_{n+1}$.
The index of an independent set in $\hat \III_{n-2}$, corresponding to $\III_{n-2}$,
is in $\hat F_{n}$.
Moreover,
if $I\in\III_n$ satisfies $\iota_n(I)\leq \fib_{k+2}$,
then $I\in \III_k$.
We also have
\begin{align*}
 \iota_n(I\cup \Set{n})
 = \iota_{n-2}(I)+\fib_{n+1}
 = \iota_{n-1}(I)+\fib_{n+1}
 = \iota_{n}(I)+\fib_{n+1}
\end{align*}
for $I \in \III_{n-2}$.

\begin{example}
 For $n=1,2,3,4$,
independent sets in $I_n$ are indexed as in the Figure \ref{fig:index}.
\begin{figure}
\begin{align*}
&
\begin{array}[t]{cc}
\III_1&F_3\\
\Set{}&1\\
\Set{1}&2\\
\end{array}
&&
\begin{array}[t]{cc}
\III_2&F_4\\
\Set{} &1\\
\Set{1} &2\\
\Set{2} &3\\
\end{array}
&&
\begin{array}[t]{cc}
\III_3&F_5\\
\Set{}&1\\
\Set{1}&2\\
\Set{2}&3\\\hline
\Set{3}&4\\
\Set{1,3}&5\\
\end{array}
&&
\begin{array}[t]{cc}
\III_4&F_6\\
\Set{}&1\\
\Set{1}&2\\
\Set{2}&3\\
\Set{3}&4\\
\Set{1,3}&5\\\hline
\Set{4}&6\\
\Set{1,4}&7\\
\Set{2,4}&8\\
\end{array}
\end{align*}
\caption{$\iota_n(I)$}
\label{fig:index}
\end{figure}
\end{example}

\begin{example}
\label{ex:iotatau:1}
It follows from direct calculation that
\begin{align*}
 &\iota_1(\tau_{1,1}(\Set{}))= \iota_1(\Set{1}) = 2,\\
 &t_{1,1}(\iota_1(\Set{}))= t_{1,1}(1)=(1,2)(1)=2, \\
 &\iota_1(\tau_{1,1}(\Set{1}))= \iota_1(\Set{}) = 1,\\
 &t_{1,1}(\iota_1(\Set{1}))= t_{1,1}(2)=(1,2)(2)=1.
\end{align*} 
 Hence
\begin{align*}
 \iota_1(\tau_{1,1}(I))=t_{1,1}(\iota_n(I))
\end{align*} 
for $I\in\III_1$.
\end{example}
\begin{example}
\label{ex:iotatau:2}
It follows from direct calculation that
\begin{align*}
 &\iota_2(\tau_{1,2}(\Set{}))= \iota_2(\Set{1}) = 2, \\
 &t_{1,2}(\iota_2(\Set{}))= t_{1,2}(1)=(1,2)(1)=2,\\
 &\iota_2(\tau_{1,2}(\Set{1}))= \iota_2(\Set{}) = 1, \\
 &t_{1,2}(\iota_2(\Set{1}))= t_{1,2}()=(1,2)(1)=1,\\
 &\iota_2(\tau_{1,2}(\Set{2}))= \iota_2(\Set{2}) = 3, \\
 &t_{1,2}(\iota_2(\Set{2}))= t_{1,2}(3)=(1,2)(3)=3.\\
\end{align*} 
Hence $\iota_2(\tau_{1,2}(I))=t_{1,2}(\iota_n(I))$.
It also follows that 
\begin{align*}
 &\iota_2(\tau_{2,2}(\Set{}))= \iota_2(\Set{2}) = 3, \\
 &t_{2,2}(\iota_2(\Set{}))= t_{2,2}(1)=(1,3)(1)=3,\\
 &\iota_2(\tau_{2,2}(\Set{1}))= \iota_2(\Set{1}) = 2, \\
 &t_{2,2}(\iota_2(\Set{1}))= t_{1,2}(2)=(1,3)(2)=2,\\
 &\iota_2(\tau_{2,2}(\Set{2}))= \iota_2(\Set{}) = 1, \\
 &t_{2,2}(\iota_2(\Set{2}))= t_{2,2}(3)=(1,3)(3)=1.\\
\end{align*} 
Hence $\iota_2(\tau_{2,2}(I))=t_{2,2}(\iota_n(I))$
for $I\in\III_2$.
\end{example}

By Examples \ref{ex:iotatau:1} and \ref{ex:iotatau:1},
we have
$\iota_n\circ\tau_{k,n}=t_{k,n}\circ \iota_n$
for $n=1,2$.
For $n>2$, we have the following lemmas:
\begin{lemma}
\label{lem:iota:n-2}
Let $n>2$ and $k<n-1$.
If 
$\iota_{n-1}\circ \tau_{k,n-1}=t_{k,n-1}\circ \iota_{n-1}$,
then $\iota_n\circ \tau_{k,n}=t_{k,n}\circ\iota_n$.
\end{lemma}
\begin{proof}
First we consider the case where 
$\iota_n(I)\leq \fib_{n+1}$.
In this case,  $\III_{n-1}$ contains $I$.
Hence
we have $\tau_{k,n}(I)=\tau_{k,n-1}(I)$,
which implies $\iota_n(\tau_{k,n}(I)) =\iota_n(\tau_{k,n-1}(I))$.
Since 
$\III_{n-1}$ contains $\tau_{k,n-1}(I)$,
we have $\iota_n(\tau_{k,n-1}(I))=\iota_{n-1}(\tau_{k,n-1}(I))$.
Since we have  
\begin{align*}
\iota_{n-1}(\tau_{k,n-1}(I))=t_{k,n-1}(\iota_{n-1}(I))=t_{k,n-1}(\iota_{n}(I)). 
\end{align*}
Since $t_{k,n}=t_{k,n-1}$,
$\iota_n(\tau_{k,n}(I)) =t_{k,n}\iota_{n}(I)$.

Next consider the case where $\fib_{n+1}<\iota_n(I)$.
We show that 
$\iota_n(\tau_{k,n}(I))$ equals
\begin{align*}
 t_{k,n}(\iota_n(I))
&=t_{k,n-1}\ttt_n t_{k,n-1}\ttt_n^{-1} (\iota_n(I))\\
&=\ttt_n t_{k,n-1}\ttt_n^{-1} (\iota_n(I))\\
&= t_{k,n-1}(\iota_n(I)-\fib_{n+1})+\fib_{n+1}.
\end{align*}
In this case we have $n\in I$.
Hence $I'=I\setminus\Set{n}\in \III_{n-1}$ and 
$\iota_n(I)=\iota(I')+\fib_{n+1}$.
It is easy to show that
\begin{align*}
 \tau_{k,n}(I)=\tau_{k,n-1}(I')\cup \Set{n}.
\end{align*}
Hence we have
\begin{align*}
\iota_n(\tau_{k,n}(I))
&=
\iota_n(\tau_{k,n-1}(I')\cup \Set{n})\\
&=
\iota_n(\tau_{k,n-1}(I'))+\fib_{n+1}\\
&=
\iota_{n-1}(\tau_{k,n-1}(I'))+\fib_{n+1}.
\end{align*}
By assumption, we have $\iota_{n-1}(\tau_{k,n-1}(I'))=t_{k,n-1}(\iota_{n-1}((I'))$,
we have
\begin{align*}
\iota_n(\tau_{k,n}(I))
&=\iota_{n-1}(\tau_{k,n-1}(I'))+\fib_{n+1}\\
&=t_{k,n-1}(\iota_{n-1}((I'))+\fib_{n+1}\\
&=t_{k,n-1}(\iota_{n-1}((I)-\fib_{n+1})+\fib_{n+1}.
\end{align*} 
\end{proof}

\begin{lemma}
\label{lem:iota:n-1}
Let $n>2$.
If 
$\iota_{n-1}\circ \tau_{n-1,n-1}=t_{n-1,n-1}\circ \iota_{n-1}$,
then $\iota_n\circ\tau_{n-1,n}=t_{n-1,n}\circ \iota_n$.
\end{lemma}
\begin{proof}
First consider the case where $\iota_n(I)\leq \fib_{n+1}$.
In this case, we have $I \in \III_{n-1}$.
Hence
we have $\tau_{n-1,n}(I)=\tau_{n-1,n-1}(I)$,
which implies $\iota_n(\tau_{n-1,n}(I)) =\iota_n(\tau_{n-1,n-1}(I))$.
Since $\tau_{n-1,n-1}(I)\in \III_{n-1}$,
we have $\iota_n(\tau_{n-1,n-1}(I))=\iota_{n-1}(\tau_{n-1,n-1}(I))$.
Hence we have  $\iota_{n-1}(\tau_{n-1,n-1}(I))=t_{n-1,n-1})(\iota_{n-1}(I))=t_{n-1,n-1}(\iota_{n}(I))$.
Since $t_{n-1,n}=t_{n-1,n-1}$,
we have
$\iota_n(\tau_{n-1,n}(I)) =t_{n-1,n}(\iota_{n}(I))$.

Next we consider the case where $\fib_{n+1}<\iota_n(I)$.
We show that 
$\iota_n(\tau_{n-1,n}(I))$ equals
$\ttt_{n-1}(\iota_n(I))=\iota_n(I)$.
In this case,
 $I$ contains $n$.
Hence it follows that
 $\tau_{n-1,n}(I)=I$
and that
$\iota_n(\tau_{n-1,n}(I))=\iota_n(I)$. 
\end{proof}

\begin{lemma}
\label{lem:iota:n}
We have
 $\iota_n\circ \tau_{n,n}=t_{n,n}\circ \iota_n$
for $n>2$
\end{lemma}
\begin{proof}
First we consider the case where $\iota_n(I)\leq \fib_{n}$.
We show that $\iota_n(\tau_{n,n}(I))$ equals
$\ttt_n(\iota_n(I))=\iota_n(I)+\fib_{n+1}$.
In this case, we have $I\in \III_{n-2}$,
which implies $I$ does not contain $n-1$, $n$.
Hence $\tau_{n,n}(I)=I\cup \Set{n}$.
Since $\iota_n(I\cup \Set{n})=\iota_{n}(I)+\fib_{n+1}$
for $I\in \III_{n-2}$,
it follows that 
$\iota_n(\tau_{n,n}(I))=\iota_n(I)+\fib_{n+1}$.

Next we consider the case where 
$\fib_{n}<\iota_n(I)\leq \fib_{n+1}$.
We show that $\iota_n(\tau_{n,n}(I))$ equals
$\ttt_n(\iota_n(I))=\iota_n(I)$.
In this case, we have $I\in \hat \III_{n-3}$,
which implies $I$ contains $n-1$.
Hence $\tau_{n,n}(I)=I$ and $\iota_n(\tau_{n,n}(I))=\iota_n(I)$.

Finaly we consider the case where 
$\fib_{n+1}<\iota_n(I)\leq \fib_{n+2}$.
We show that $\iota_n(\tau_{n,n}(I))$ equals
$\ttt_n(\iota_n(I))=\iota_n(I)-\fib_{n+1}$.
In this case, we have $I\in \hat\III_{n-2}$,
which implies $I$ contains $n$.
Hence $\tau_{n,n}(I)=I\setminus\Set{n}$.
Since $\iota_n(I\cup \Set{n})=\iota_{n}(I)+\fib_{n+1}$
for $I\in \III_{n-2}$,
it follows that 
$\iota_n(\tau_{n,n}(I))=\iota_n(I)-\fib_{n+1}$.
\end{proof}

Lemmas \ref{lem:iota:n-2}, \ref{lem:iota:n-1} and \ref{lem:iota:n}
imply the following lemma:
\begin{lemma}
If 
$n>2$ and
$\iota_{n-1}\circ \tau_{k,n-1}=t_{k,n-1}\circ \iota_{n-1}$,
then $\iota_n\circ\tau_{k,n}=t_{k,n}\circ \iota_n$.
\end{lemma}
Therefore, by induction on $n$,
we have 
$\iota_n\circ\tau_{k,n}=t_{k,n}\circ \iota_n$
for $n\geq 1$.

Since $\iota_n$ induces a bijection between $\III_{n}$ and $F_{n+2}$,
the bijective map 
\begin{align*}
 \Set{\tau_{1,n},\ldots,\tau_{n,n}} &\to G_n \\
 \tau_{k,n} &\mapsto t_{k,n}
\end{align*}
induces the homomorphism from $\Gamma_n$ to $\Braket{G_n}$
which preserves the actions.
Since $G_n$ generates the $\fib_{n+2}$-th symmetric group $\Sym_{\fib_{n+2}}$,
the group $\Gamma_n$ is isomorphic to  $\Sym_{\fib_{n+2}}$.
Hence $|\Gamma_n|=|\Sym_{\fib_{n+2}}|=|\Sym_{\III_n}|$.
Since $\Gamma_n\subset\Sym_{\III_n}$,
 we have  Theorem \ref{thm:togglegroupissymmetricgroup}.

\bibliography{by-mr}
\bibliographystyle{amsplain-url} 

\end{document}